\documentclass[12pt, reqno]{amsart}

\usepackage{verbatim}
\usepackage{fullpage}
\usepackage{amsfonts}
\usepackage{enumerate}
\usepackage{todonotes}
\usepackage[noadjust]{cite}
\usepackage{graphicx}
\usepackage{amsmath, amsthm, amssymb, bbm}
\usepackage{fancybox}
\usepackage{xcolor}

\usepackage{tikz-cd}

\newcommand\iprec{\mathrel{\ooalign{$\prec$\cr
 \,\raise0.85ex\hbox{\scriptsize$\circ$}\cr}}}

\newcommand{\R}{\mathbb{R}}

\newcommand{\Z}{\mathbb{Z}}

\newcommand{\1}{\mathbbm{1}}

\DeclareMathOperator{\diag}{diag}

\DeclareMathOperator{\diam}{diam}

\newcommand{\qn}{\rho_A}

\NewDocumentCommand\TLseq{O{\alpha}}{\mathbf{f}^{#1}_{p,q}}
\newcommand{\TLfA}{\dot{\mathbf{F}}^{\alpha}_{p,q}(A)}
\newcommand{\TLfB}{\dot{\mathbf{F}}^{\alpha}_{p,q}(B)}
\newcommand{\TLfAi}{\dot{\mathbf{F}}^{\alpha}_{\infty,q}(A)}
\newcommand{\TLfAii}{\dot{\mathbf{F}}^{\alpha}_{\infty,\infty}(A)}
\newcommand{\TLsA}{\dot{\mathbf{f}}^{\alpha}_{p,q}(A)}
\newcommand{\TLsB}{\dot{\mathbf{f}}^{\alpha}_{p,q}(B)}
\newcommand{\TLAi}{\dot{\mathbf{f}}^{\alpha}_{\infty,q}(A)}

\newcounter{Theorem}

\numberwithin{equation}{section}
\numberwithin{Theorem}{section}

\theoremstyle{plain} 
\newtheorem{thm}[Theorem]{Theorem}

\newtheorem{lem}[Theorem]{Lemma}

\theoremstyle{definition}
\newtheorem{defn}[Theorem]{Definition}

\theoremstyle{remark}

\title{Triebel-Lizorkin spaces and expansive matrices}

\author{Marcin Bownik}
\address{Department of Mathematics, University of Oregon, Eugene, OR 97403--1222, USA}
\email{mbownik@uoregon.edu}

\author{Jordy Timo van Velthoven}
\address{Faculty of Mathematics,
University of Vienna, 
Oskar-Morgenstern-Platz 1,
1090 Vienna, Austria}
\email{jordy-timo.van-velthoven@univie.ac.at}

\makeatletter
\@namedef{subjclassname@2020}{\textup{2020} Mathematics Subject Classification}
\makeatother

\subjclass[2020]{42B35, 46E35}

\keywords{Expansive matrices, quasi-norms, Triebel-Lizorkin spaces, sequence spaces}

\begin{document}

\maketitle

\begin{center}
    \emph{Dedicated to Hans Triebel on the occasion of his 90th birthday}
\end{center}

\begin{abstract}
We survey recent classification theorems for expansive matrices that generate the same anisotropic homogeneous Triebel-Lizorkin function space or sequence space. The function spaces are classified precisely by those matrices for which their associated homogeneous quasi-norms on Euclidean space are  equivalent, whereas the sequence spaces are classified by a strictly stronger condition. We unravel this discrepancy between function spaces and sequence spaces by showing that two sequence spaces are retracts of each other whenever the corresponding function spaces are the same.
\end{abstract}

\section{Introduction}
Besov and Triebel-Lizorkin spaces provide a unifying framework for various classical function spaces in harmonic analysis, such as BMO spaces, Hardy spaces, Lebesgue spaces, Lipschitz spaces and Sobolev spaces. The unifying perspective that these function spaces have in common is their characterization in terms of a Littlewood-Paley decomposition. In \cite{frazier1985, frazier1990discrete}, Frazier and Jawerth used this characterization to reduce the study of various key properties of Besov and Triebel-Lizorkin spaces to the sequence spaces appearing in the discrete Littlewood-Paley decomposition. The key technical tool in the papers \cite{frazier1985, frazier1990discrete} is the so-called $\varphi$-transform $S_{\varphi}$ associated to a function $\varphi$, which maps a distribution $f$ to a sequence of coefficients $\{ \langle f, \varphi_Q \rangle \}_{Q \in \mathcal{Q}}$, where $\mathcal{Q}$ denotes the set of all dyadic cubes and $\varphi_{Q}$ denotes a translate and dilate of $\varphi$ relative to a dyadic cube $Q$. The inverse $\varphi$-transform $T_{\psi}$ associated to a function $\psi$  maps a sequence $c = \{c_Q \}_{Q \in \mathcal{Q}}$ of complex numbers to $\sum_{Q \in \mathcal{Q}} c_Q \psi_Q$. Under appropriate conditions on the pair $(\varphi, \psi)$, a basic result in \cite{frazier1985, frazier1990discrete} shows the decomposition $f = T_{\psi} S_{\varphi} f$ of any adequate distribution $f$  and shows that membership of $f$ in a Besov space $\dot{\mathbf{B}}_{p,q}^{\alpha}$ or a Triebel-Lizorkin space $\dot{\mathbf{F}}_{p,q}^{\alpha}$ is completely determined by $S_{\varphi} f$ belonging to an associated Besov sequence space $\dot{\mathbf{b}}_{p,q}^{\alpha}$ or Triebel-Lizorkin sequence space $\dot{\mathbf{f}}_{p,q}^{\alpha}$. More precisely, the following diagram commutes (and similarly for Besov spaces).
\begin{center}
\begin{tikzcd}[column sep=small]
& \dot{\mathbf{f}}_{p,q}^{\alpha} \arrow[dr, "T_\psi"]& \\
\dot{\mathbf{F}}_{p,q}^{\alpha} \arrow[ur, hook, "S_\varphi"] \arrow[rr, "id"] &  & \dot{\mathbf{F}}_{p,q}^{\alpha} \\
\end{tikzcd}
\end{center}

The $\varphi$-transform theory of Frazier and Jawerth \cite{frazier1985, frazier1990discrete} has been extended to anisotropic Besov spaces $\dot{\mathbf{B}}_{p,q}^{\alpha} (A)$ and Triebel-Lizorkin spaces $\dot{\mathbf{F}}_{p,q}^{\alpha} (A)$ associated to general expansive dilations by the first author in  \cite{bownik2005atomic, bownik2006atomic, bownik2007anisotropic, bownik2008duality}.
An interesting question on such anisotropic spaces is how these spaces depend on the choice of dilation. This question was settled for anisotropic Hardy spaces by the first author \cite{bownik2003anisotropic} and was settled for anisotropic Triebel-Lizorkin sequence spaces associated to diagonal matrices with positive anisotropy by Triebel \cite{triebel2004wavelet}; see also \cite[Section 5.3]{triebel2006theory}. More recently, such classification theorems have also been settled for anisotropic Besov spaces \cite{cheshmavar2020classification}, anisotropic Triebel-Lizorkin spaces \cite{koppensteiner2024classification, velthoven2024classification} and anisotropic Triebel-Lizorkin sequence spaces \cite{velthoven2026discrete}. For homogeneous function spaces, these classification theorems show that two function spaces associated to expansive matrices $A$ and $B$ coincide (i.e., $\TLfA = \dot{\mathbf{F}}_{p,q}^{\alpha} (B)$) for a generic choice of parameters $(p,q,\alpha)$ precisely when two associated homogeneous quasi-norms $\rho_A , \rho_B : \mathbb{R}^d \to [0, \infty)$ are equivalent (see Section \ref{sec:function}). On the other hand, the coincidence of sequence spaces (i.e., $\TLsA = \dot{\mathbf{f}}_{p,q}^{\alpha} (B)$) is shown to be characterized by the condition that the set $\{A^{-j} B^{j} : j \in \mathbb{Z} \}$ is finite (see Section \ref{sec:sequence}), which is a strictly stronger condition than two homogeneous quasi-norms $\rho_A$ and $\rho_B$ being equivalent.

The aim of the present paper is to survey the aforementioned classification theorems for (homogeneous) anisotropic Triebel-Lizorkin spaces and to unravel the somewhat surprising fact that two function spaces can coincide even when their associated sequence spaces are different. In particular, we show the existence of a natural isomorphism between the spaces $\TLsA$ and $\TLsB$ when $|\det (A)| = |\det (B)|$ and $\TLfA = \TLfB$. The isomorphism has the form of permutation operators, which are induced by bounded displacement bijections between lattices $A^{j}\Z^d$ and $B^{j}\Z^d$ for all scales $j\in \Z$. 
In the general case when $\TLfA = \TLfB$, but possibly $|\det (A)| \neq |\det (B)|$, we show that the spaces $\TLsA$ and $\TLsB$ are retracts of each other, in the sense that the following commutative diagram (for the case $|\det(A)| > |\det(B)|$) holds for adequate operators $S$ and $T$ induced by the displacement injections between lattices $A^{j}\Z^d$ and $B^{i}\Z^d$ for appropriate choice of scales $i=i(j)\in \Z$. We refer to Section \ref{sec:sequence2} for precise details.

\begin{center}
\begin{tikzcd}[column sep=small]
& \TLsB \arrow[dr, "T"]& \\
\TLsA \arrow[ur, hook, "S"] \arrow[rr, "id"] &  & \TLsA\\
\end{tikzcd}
\end{center}

The organization of this paper is as follows. Section \ref{sec:expansive} provides background on expansive matrices and associated homogeneous quasi-norms. The classification theorem for (homogeneous) anisotropic Triebel-Lizorkin function spaces and sequence spaces are discussed in Section \ref{sec:function} and Section \ref{sec:sequence}. Lastly, our new results on the equivalence of sequence spaces are proven in Section \ref{sec:sequence2}.

\subsection*{Notation}
Given two functions $f_1, f_2 : X \to [0, \infty)$ on a set $X$, we write $f_1 \lesssim f_2$ if there exists $C>0$ such that $f_1 (x) \leq C f_2 (x)$ for all $x \in X$. We write $f_1 \asymp f_2$ whenever $f_1 \lesssim f_2$ and $f_2 \lesssim f_1$. 

The Euclidean norm of a vector $x \in \mathbb{R}^d$ will be denoted by $\| x \|$ and for a matrix $A \in \mathbb{R}^{d \times d}$, we write $\|A\|$ for its operator norm. The Lebesgue measure of a measurable set $U \subseteq \mathbb{R}^d$ will be denoted by $|U|$.

\section{Expansive matrices and quasi-norms} \label{sec:expansive}
We start by recalling various basic facts on expansive matrices and associated homogeneous quasi-norms.
We refer to \cite[Chapter 1, Section 2]{bownik2003anisotropic} for further details and properties.

A matrix $A \in \mathrm{GL}(d, \mathbb{R})$ is said to be \emph{expansive} if all its eigenvalues $\lambda \in \sigma(A)$ satisfy $|\lambda|>1$. Equivalently, $A \in \mathrm{GL}(d, \mathbb{R})$ is expansive if $\| A^{-j} \| \to 0$ as $j \to \infty$.

A \emph{homogeneous quasi-norm} associated with an expansive
matrix $A$ is a Borel measurable function $\rho_A: \R^d  \to
[0,\infty)$ satisfying the conditions:
\begin{enumerate}[(i)]
\item $ \rho_A(x) > 0$ for $x \neq 0$;
\item $\rho_A(Ax) = |\det(A)|  \rho_A(x)$ for every $x \in \R^d$;
\item there exists $c>0$ such that $ \rho_A(x+y) \leq
c(\rho_A(x)+\rho_A(y))$ for every $x, y \in \R^d$.
\end{enumerate}
A quasi-norm $\rho_A$ satisfying $\rho_A(x)=\rho_A(-x)$ for all $x\in \R^d$ always exists. Any two quasi-norms $\rho_A, \rho'_A$ associated to an expansive matrix $A$ are equivalent, in the sense that $\rho_A \asymp \rho'_A$. In addition, there exists $c > 0$ such that, for every $x \in \mathbb{R}^d$,
\begin{align*}
 c^{-1} \rho_A (x)^{\frac{\ln \lambda_- }{ \ln |\det(A)|}} \leq \| x \| \leq c \rho_A (x)^{\frac{\ln \lambda_+ }{ \ln |\det(A)|}}, \quad \rho_A (x) \geq 1, \\
 c^{-1} \rho_A (x)^{\frac{\ln \lambda_+ }{ \ln |\det(A)|}} \leq \| x \| \leq c \rho_A (x)^{\frac{\ln \lambda_- }{ \ln |\det(A)|}}, \quad \rho_A (x) \leq 1,
\end{align*}
where $\lambda_-, \lambda_+ \in (0, \infty)$ are constants such that $ \lambda_- < \min_{\lambda \in \sigma(A)} |\lambda| \leq \max_{\lambda \in \sigma(A)} |\lambda| < \lambda_+$.

We say that two expansive matrices $A$ and $B$ are \emph{equivalent} whenever two associated homogeneous quasi-norms $\rho_A$ and $\rho_B$ are equivalent. The following lemma provides a characterization of this equivalence, cf. \cite[Lemma 10.2]{bownik2003anisotropic}. 

\begin{lem}[\cite{bownik2003anisotropic}] \label{lem:equivalent}
Let $A, B \in \mathrm{GL}(d, \mathbb{R})$ be expansive and let $\rho_A, \rho_B$ be associated homogeneous quasi-norms. Then $\rho_A$ and $\rho_B$ are equivalent if and only if
\[
\sup_{j \in \mathbb{Z}} \| A^{-j} B^{\lfloor \varepsilon j \rfloor} \| < \infty,
\]
where $\varepsilon := \ln |\det(A)| / \ln |\det(B)|$. 
\end{lem}

We mention that if $A, B \in \mathrm{GL}(d, \mathbb{R})$ are expansive matrices having only positive eigenvalues and satisfy $\det(A) = \det (B)$, then $A$ and $B$ are equivalent precisely whenever $A=B$, cf. \cite[Theorem 7.9]{cheshmavar2020classification}. 

Euclidean space $\mathbb{R}^d$ equipped with a homogeneous quasi-norm $\rho_A$ and Lebesgue measure $\mu$ forms a space of homogeneous type in the sense of Coifman-Weiss \cite{coifman1977extensions}. 
Given a homogeneous quasi-norm $\rho_A$,  we define the associated $\rho_A$-ball by
\[
B_{\qn}(x,r)= \big\{ y\in\R^d: \qn(x-y)<r \big \}, \qquad x\in\R^d,\ r>0,
\]
and let $\mathcal{B}$ denote the collection of all such balls. 
Let $\mathcal Q$ be the collection of all {\it dilated cubes}
\[
{ \mathcal Q} = \big\{Q_{j,k} :=A^j([0,1]^d+k): j\in\Z,\ k\in\Z^d \big\}
\]
adapted to the
action of a dilation $A$. Obviously, if $A=2 \cdot I_d$, then we obtain the usual collection of {\it dyadic cubes}. Let
\[
x_Q=A^j k, \qquad Q=A^j([0,1]^d+k) \in\mathcal Q,
\]
be the ``lower-left corner'' of $Q$. 
The scale of a ball $B=B_{\qn}(x_0,r)\in \mathcal B$ is defined as 
\[
\operatorname{scale}(B) = \lfloor \log_{|\det(A)|} r\rfloor.
\]
The scale of a dilated cube $Q=A^j([0,1]^d+k) \in \mathcal Q$ is defined as  $\operatorname{scale}(Q)=j$. Alternatively, $\operatorname{scale}(Q) = \log_{|\det(A)|}|Q|$.

\section{Function spaces} \label{sec:function}
We follow the definition of the anisotropic Triebel-Lizorkin function spaces in \cite{bownik2006atomic,  bownik2007anisotropic}.

For an expansive $A \in \mathrm{GL}(d, \mathbb{R})$, let $\varphi \in \mathcal{S}(\mathbb{R}^d)$ be a Schwartz function with Fourier transform $\widehat{\varphi} \in C_c^{\infty} (\mathbb{R}^d \setminus \{0\})$ satisfying
\[
 \sup_{j \in \mathbb{Z}} |\widehat{\varphi} ((A^t)^j \xi)| > 0 \quad \text{for all} \quad \xi \in \mathbb{R}^d \setminus \{0\}
\]
and let $\varphi_j := |\det(A)|^j \varphi (A^j \cdot)$ for $j \in \mathbb{Z}$.
For $\alpha \in \mathbb{R}$ and $p,q \in (0, \infty]$, the associated (homogeneous) anisotropic Triebel-Lizorkin space $\TLfA$ is defined as the space of all tempered distributions $f \in \mathcal{S}'(\mathbb{R}^d)  / \mathcal{P}$ (modulo polynomials) satisfying
\[
\| f \|_{\TLfA} := \bigg\| \bigg( \sum_{j \in \mathbb{Z}} \big( |\det(A)|^{\alpha j} |f \ast \varphi_j | \big)^q \bigg)^{1/q} \bigg\|_{L^p} < \infty
\]
whenever $p \in (0, \infty)$ (with the usual modification for $q = \infty$), and
\[
 \| f \|_{\TLfAi} := \sup_{\ell \in \mathbb{Z}, k \in \mathbb{Z}^d} \bigg(
 \frac{1}{|\det(A)|^{\ell}} \int_{A^{\ell} ([0,1]^d+k)} \sum_{j = \ell}^{\infty} \big(|\det(A)|^{\alpha j} |(f \ast \varphi_j) (x)|)^q \; dx \bigg)^{1/q} < \infty
\]
or $\| f \|_{\TLfAii} := \sup_{j \in \mathbb{Z}} |\det(A)|^{\alpha j} \| f \ast \varphi_j \|_{L^{\infty}} < \infty$ whenever $p = \infty$. The space $\TLfA$ is complete with respect to the quasi-norm $\| \cdot \|_{\TLfA}$ and independent of the choice of the defining vector $\varphi$, with equivalent quasi-norms for different choices, cf. \cite{bownik2006atomic,  bownik2007anisotropic}.

The scale of anistropic Triebel-Lizorkin spaces $\TLfA$ contains, among others, the anisotropic Hardy spaces introduced in \cite{bownik2003anisotropic}. For $p \in (0, 1]$, the anisotropic Hardy space is defined as
\[
 H^p (A) := \bigg\{ f \in \mathcal{S}' (\mathbb{R}^d) : M_{\phi}^0 f \in L^p (\mathbb{R}^d) \bigg\},
\]
where $\phi \in \mathcal{S}(\mathbb{R}^d)$ is a fixed function with $\int_{\mathbb{R}^d} \phi(x) \;dx \neq 0$, and
\[
 M_{\phi}^0 f (x) = \sup_{j \in \mathbb{Z}} |f \ast \phi_j (x)|, \quad x \in \mathbb{R}^d,
\]
is the \emph{radial maximal function} of $f$ associated to $\phi$. See \cite[Chapter 1]{bownik2003anisotropic} for further details and various characterizations.

In \cite[Chapter 1, Section 10]{bownik2003anisotropic}, the following classification of dilations $A$ yielding the same anisotropic Hardy space $H^p (A)$ was proven, cf. \cite[Theorem 10.5]{bownik2003anisotropic}.

\begin{thm}[\cite{bownik2003anisotropic}]
Let $A, B \in \mathrm{GL}(d, \mathbb{R})$ be expansive. The following are equivalent:
\begin{enumerate}[(i)]
 \item $A$ and $B$ are equivalent.
 \item $H^p(A) = H^p(B)$ for some $p \in (0, 1]$.
 \item $H^p(A) = H^p(B)$ for all $p \in (0, 1]$.
\end{enumerate}
\end{thm}

The anisotropic Hardy spaces $H^p(A)$ can be identified with the anisotropic Triebel-Lizorkin spaces $\dot{\mathbf{F}}^0_{p, 2} (A)$ for $p \in (0, 1]$. For $p \in (1, \infty)$, the space $\dot{\mathbf{F}}^0_{p, 2} (A)$ can be identified with the usual Lebesgue space $L^p (\mathbb{R}^d)$, and thus does not depend on the choice of $A$. It turns out that this is the only range of Triebel-Lizorkin spaces that is independent of the dilation. The following classification theorem was shown in \cite{koppensteiner2024classification}.

\begin{thm}[\cite{koppensteiner2024classification}]
Let $A, B \in \mathrm{GL}(d, \mathbb{R})$ be expansive. The following are equivalent:
\begin{enumerate}
 \item $A$ and $B$ are equivalent.
 \item $\TLfA = \TLfB$ for some $(\alpha, p, q) \in \mathbb{R} \times (0, \infty]^2$ with $(\alpha, p, q) \notin \{0\} \times (1, \infty) \times \{2\}$.
 \item  $\TLfA = \TLfB$ for all $(\alpha, p, q) \in \mathbb{R} \times (0, \infty]^2$.
\end{enumerate}
\end{thm}

A classification of \emph{inhomogeneous} anisotropic Triebel-Lizorkin spaces was obtained in \cite{velthoven2024classification}.

\section{Sequence spaces} \label{sec:sequence}
Let $A \in \mathrm{GL}(d, \mathbb{R})$ be expansive.
For $\alpha \in \R$ and $p,q\in (0,\infty]$, the (homogeneous) anisotropic discrete Triebel-Lizorkin sequence space
$\dot {\mathbf f}^\alpha_{p,q}(A)$  is defined as the space of all $c\in \mathbb C^{\Z \times \Z^d}$ satisfying
\[
||c||_{\dot {\mathbf f}^\alpha_{p,q}(A)}:= \bigg\| \bigg( \sum_{j\in \Z} \sum_{k\in \Z^d} ( |\det(A)|^{-j(\alpha+1/2)}|c_{j,k}| \1_{A^j ([0, 1]^d + k)})^q \bigg)^{1/q} \bigg\|_{L^p} <\infty
\]
whenever $p < \infty$ (with the usual modification for $q = \infty)$, and
\[
 \| c \|_{\TLAi} := \sup_{P \in \mathcal{Q}} \bigg(\frac{1}{|P|} \int_P \sum_{\substack{ j \in \mathbb{Z} \\ j \leq \operatorname{scale}(P)}} \sum_{k \in \mathbb{Z}^d} \big( |\det(A)|^{-j (\alpha+1/2)} |c_{j,k}| \1_{A^j ([0,1]^d + k)} (x) \big)^q \; dx \bigg)^{1/q} < \infty,
\]
where the case $q=\infty$ has to be interpreted as
$ \sup_{j \in \mathbb{Z}, k \in \mathbb{Z}^d} |\det(A)|^{-j (\alpha + 1/2)} |c_{j,k} | < \infty.
$

The following result \cite[Corolary 3.5]{bownik2008duality} bridges a gap in the definition of $\TLsA$ spaces between the generic case $p<\infty$ and the special case $p=\infty$.

\begin{thm}\label{pinf}
Suppose that $\alpha \in \R$ and $0<p,q \le \infty$. Fix $0<\eta<1$. Then, for any sequence $c\in \mathbb C^{\Z \times \Z^d}$
\[
||c||_{\TLsA} \asymp \inf \bigg\| \bigg( \sum_{j\in \Z} \sum_{k\in \Z^d} ( |\det(A)|^{-j(\alpha+1/2)}|c_{j,k}| \1_{E_{j,k}})^q \bigg)^{1/q} \bigg\|_{L^p} 
\]
where infimum is taken over all Lebesgue measurable sets $E_{j,k} \subset A^j ([0, 1]^d + k)$ satisfying  $|E_{j,k}| > \eta|\det A|^j$ for $j\in \Z$, $k\in \Z^d$.
\end{thm}

It is readily verified from the definitions of the quasi-norms defining Triebel-Lizorkin sequences spaces that if $A, B \in \mathrm{GL}(d, \mathbb{R})$ are expansive and $\alpha \in \mathbb{R}$, $p, q \in (0, \infty]$ are such that $p=q$ and
\[
 |\det(A)|^{\alpha + 1/2 - 1/p} = |\det(B)|^{\alpha + 1/2 - 1/p},
\]
then $\TLsA = \TLsB$. Thus, the particular range of anisotropic Triebel-Lizorkin sequence spaces corresponding to $p=q$ is \emph{independent} of the choice of expansive matrix. This is merely a reflection of two facts: 
\begin{enumerate}
\item Triebel-Lizorkin sequence spaces coincide with Besov sequence spaces when $p=q$, and
\item Besov sequence spaces ${\dot{\mathbf{b}}^{\alpha}_{p,q}(A)}$ are independent of the choice of a dilation $A$ (up to a normalization of the smoothness parameter $\alpha$) in light of Triebel's transference method \cite[Proposition 5.26]{triebel2006theory}.
\end{enumerate}
The question whether the full scale of anisotropic Triebel-Lizorkin spaces is independent of the choice of dilation matrix was also settled by Triebel \cite{triebel2006theory} for the case of diagonal matrices. The following theorem is \cite[Proposition 5.26]{triebel2006theory}.

\begin{thm}[\cite{triebel2006theory}]
Let $\alpha \in \mathbb{R}$, $p \in (0, \infty)$ and $q \in (0, \infty]$. Let $A = \diag(2^{a_1}, ..., 2^{a_d})$ and $B = \diag(2^{b_1}, ..., 2^{b_d})$ for anisotropies
\[
(a_1, ..., a_d), (b_1, ..., b_d) \in (0, \infty)^d \quad \text{satisfying} \quad \sum_{i = 1}^d a_i = \sum_{i =1}^d b_i = d.
\]
Suppose that $A \neq B$. Then $\TLsA = \TLsB$ if and only if $p = q$. 
\end{thm} 

Recently, a full classification of (possibly nondiagonal) dilations yielding anisotropic Triebel-Lizorkin  sequence spaces was obtained in \cite{velthoven2026discrete}. The precise statement is as follows.

\begin{thm}[\cite{velthoven2026discrete}] \label{thm:classification_sequence}
Let $A, B \in \mathrm{GL}(d, \mathbb{R})$ be expansive. Then the following are equivalent:
\begin{enumerate}[(i)]
 \item The set $\{A^j B^{-j} : j \in \Z \}$ is finite.
 \item $\TLsA = \TLsB$ for some $(\alpha, p, q) \in \mathbb{R} \times (0, \infty]^2$ with $p\neq q$.
 \item $\TLsA = \TLsB$ for all $(\alpha, p, q) \in \mathbb{R} \times (0, \infty]^2$.
\end{enumerate}
\end{thm}

Note that condition (i) of Theorem \ref{thm:classification_sequence} implies, in particular, that $|\det(A)| = |\det(B)|$, and hence that the matrices $A$ and $B$ are equivalent.
As such, the condition that classifies Triebel-Lizorkin sequences spaces is stronger than the corresponding condition for Triebel-Lizorkin function spaces.

\section{New results on sequence spaces} \label{sec:sequence2}
This section contains our new results on the equivalence of sequence spaces.

\subsection{Bounded displacement maps}
The following result on bounded displacement maps is part of the folklore; see, e.g., 
 \cite[Section 3.2]{kolountzakis1997multi} or (the proof of) \cite[Proposition 2.1]{haynes2014equivalence} for the statement on bijections. However, as we need this statement also for lattices with possibly unequal covolume, we provide its proof for the sake of completeness.

\begin{lem} \label{bd}
Suppose that $S, T \in \mathrm{GL}(d, \mathbb{R})$ are such that $|\det (S)| \geq |\det (T)|$. Let $\Lambda_S = S\Z^d$ and $\Lambda_T=T \Z^d$ be the corresponding lattices with bounded fundamental domains $F_S$ and $F_T$, respectively. Let $r_S := \diam(F_S)$ and $r_T := \diam(F_T)$.
Then there exists an injection $\phi: \Lambda_S \to \Lambda_T$ satisfying
\[
\sup_{x \in \Lambda_S} \| x-\phi(x) \| \leq r_T + r_S.
\]
Moreover, if $|\det(S)| = |\det(T)|$, then $\phi : \Lambda_S \to \Lambda_T$ can be chosen to be a bijection. 
\end{lem}
\begin{proof}
For each $x\in \Lambda_S$, define the finite set
\[
U_x := \{\, y \in \Lambda_T : (x+F_S)\cap (y+F_T)\neq \emptyset \,\}.
\]
Note that if $y\in U_x$, then there exists $z\in (x+F_S)\cap (y+F_T)$, and thus
\[
\|x-y\| \le \|x-z\| + \|z-y\| \le r_S + r_T.
\]
For settling the claim, it suffices therefore to find an injective map $\phi:\Lambda_S\to\Lambda_T$ with $\phi(x)\in U_x$ for all $x \in \Lambda_S$. For this, we verify Hall's condition \cite{rado1967note}.

Let $E \subseteq \Lambda_S$ be finite and set
\[
N(E) := \bigcup_{x\in E} U_x = \bigg\{\, y\in \Lambda_T : (y+F_T)\cap \bigcup_{x\in E}(x+F_S)\neq \emptyset \,\bigg\}.
\]
Define
$
V := \bigcup_{x\in E}(x+F_S)$ and $ W := \bigcup_{y\in N(E)}(y+F_T).
$
Since the union $\mathbb{R}^d = \bigcup_{x \in \Lambda_S} x+F_S$ is disjoint, we have that $|V|=|E| |F_S|=|E| |\det (S)|$. A similar argument shows that $|W|=|N(E)| |F_T|=|N(E)| |\det (T)|$. If $z\in V$, then $z\in x+F_S$ for some $x\in E$. Moreover, since $\{y+F_T : y\in \Lambda_T\}$ tiles $\mathbb{R}^d$, there exists $y\in \Lambda_T$ such that $z\in y+F_T$. Hence,  $z\in (x+F_S)\cap (y+F_T)$, that is, $y\in U_x\subseteq N(E)$. This also implies that  $z\in W$. Thus $V\subseteq W$, which yields that
\[
|E| |\det (S)| \le |N(E)| |\det (T)|.
\]
Using $|\det (S)|\ge |\det (T)|$, it follows that $|E|\le |N(E)|$, which verifies Hall's condition. Therefore, by Hall's theorem (see, e.g., \cite[Theorem 1]{rado1967note}), there exists an injective map $\phi:\Lambda_S\to\Lambda_T$ such that $\phi(x)\in U_x$ for all $x\in \Lambda_S$, which settles the first claim. 

For the additional part, note that if $|\det(S)| = |\det(T)|$, then we also get an injection $\psi : \Lambda_T \to \Lambda_S$ satisfying
\[
\sup_{y \in \Lambda_T} \| y - \psi (y) \| \leq r_T + r_S.
\]
Using (the proof of) the Schr\"oder-Bernstein theorem, we obtain a bijection with the claimed properties, see, e.g., \cite[Section 3.2]{kolountzakis1997multi} for more details.
\end{proof}

The following simple consequence of Lemma \ref{bd} is what we will actually use.

\begin{lem}\label{bd5}
Suppose that $A, B \in \mathrm{GL}(d, \mathbb{R})$ are such that $|\det (A)|\geq |\det (B)|$. Then there exists an injection $\pi: \mathbb{Z}^d \to \mathbb{Z}^d$ satisfying
\[
\sup_{k \in \mathbb{Z}^d} \| k - A^{-1} B\pi(k) \| \le \sqrt{d} (1+ ||A^{-1} B||).
\]
Moreover, if $|\det(A)| = |\det(B)|$, then $\pi$ can be chosen to be a bijection.
\end{lem}
\begin{proof}
We aim to apply Lemma \ref{bd} to the matrices $S := I_d$ and  $T := A^{-1} B$. For the fundamental domains $F_S := (0,1]^d$ and $F_T := T(0,1]^d$, we have have that $\diam(F_S) \leq \sqrt{d}$ and  $\diam(F_T) \leq \sqrt{d} \|T\|$. Since $|\det(S)| \geq |\det(T)|$ (resp. $|\det(S)| = |\det(T)|$), it follows therefore from Lemma \ref{bd} that there exists an injection (resp. a bijection) $\phi : \mathbb{Z}^d \to T \mathbb{Z}^d$ such that $\sup_{k \in \mathbb{Z}^d} \| k - \phi(k) \| \leq  \sqrt{d} (1+\| T \|)$. Setting $\pi (k) := T^{-1} \phi (k)$ yields thus an injection (resp. a bijection) $\pi : \mathbb{Z}^d \to \mathbb{Z}^d$ satisfying
\[
\sup_{k \in \mathbb{Z}^d} \| k - T \pi(k) \| = \sup_{k \in \mathbb{Z}^d} \| k -  \phi(k) \| \leq  \sqrt{d} (1+ ||T||),
\]
which settles the claim.
\end{proof}

\subsection{Majorizing and maximal functions}
This subsection contains some further auxiliary results that will be used in proving our new results on sequence spaces.

We start with the following definition of majorizing functions.

\begin{defn}
Let $A \in \mathrm{GL}(d, \mathbb{R})$ be expansive, $0<r<\infty$, and $\lambda>0$. Given $c=(c_{j,k})_{j\in \Z, k\in \Z^d} \in \mathbb{C}^{\Z \times \Z^d}$, define an associated sequence $(c^A_j)_{j\in \Z}$ of functions  $c_j^A : \mathbb{R}^d \to [0, \infty)$ by
\begin{equation}\label{cja}
c^{A}_j(x) = \bigg( \sum_{k\in \Z^d} \frac{|c_{j,k}|^r}{(1+\rho_A(A^{-j}x - k))^\lambda} \bigg)^{1/r}, \quad x \in \R^d.
\end{equation}
\end{defn}

We have that  
\begin{align} \label{main5}
\sum_{k\in \Z^d} |c_{j,k}| \1_{A^j ([0, 1]^d + k)}(x) \lesssim c^A_j(x) \quad \text{for any} \quad x \in \mathbb{R}^d.
\end{align}
Indeed, note that for any $x\in A^j ([0, 1]^d + k)$, we have $\rho_A(A^{-j}x - k) \le \sup_{y\in [0,1]^d} \rho_A(y) \le C$ for some $C>0$, and thus
\begin{align*}
\sum_{k\in \Z^d} |c_{j,k}| \1_{A^j ([0, 1]^d + k)}(x) &\le
\sum_{k\in \Z^d} \frac{|c_{j,k}|(1+C^\lambda)^{1/r} }{(1+\rho_A(A^{-j}x - k)^{\lambda})^{1/r} }  \1_{A^j ([0, 1]^d + k)}(x) \\
&\le (1+C^\lambda)^{1/r}  c^A_j(x),
\end{align*}
which shows \eqref{main5}. 

In addition to the majorizing functions defined above, we will use the Hardy-Littlewood maximal function. Recall that, for a Borel measurable function $f : \mathbb{R}^d \to \mathbb{C}$,  its Hardy-Littlewood maximal function $M_{\qn}f$ with respect to the Lebesgue measure is given by
\[
M_{\qn}f(x) = \sup_{x\in B \in \mathcal B} \frac{1}{|B|} \int_B |f(y)| \; d\mu(y), \quad x \in \mathbb{R}^d.
\]
The following lemma is a variation on that of \cite[Lemma 6.2]{bownik2006atomic}.

\begin{lem}\label{max} 
Let $0<a\le r<\infty$, $\lambda>r/a$ and $0 < \eta < 1$.  Let $E_{j,k} \subseteq Q_{j,k}=A^j([0,1]^d+k)$ be Lebesgue measurable sets satisfying $|E_{j,k}| > \eta |\det(A)|^j$ for $j \in \Z$ and $k \in \Z^d$.

There exists $C > 0$ with the following property:
If  $c \in \mathbb{C}^{ \Z \times \Z^d} $ and $j\in \Z$, then
\begin{equation}\label{max1}
c^A_j(x) \le C \bigg( M_{\rho_A} \bigg( \sum_{k\in \Z^d} |c_{j,k}| \1_{E_{j,k}}\bigg)^a \bigg)^{1/a} \quad \text{for all} \quad x\in \R^d.
\end{equation}
The constant $C$ depends only on $\lambda-r/a$, $\eta^{-1/a}$ and the choice of dilation $A$.
\end{lem}
\begin{proof}
 Let $x \in \R^d$ be arbitrary. For $m \in \mathbb{N}$, define the sets
 \[
  E_0 (x) := \big\{ k \in \Z^d : \rho_A (A^{-j} x - k) \leq 1 \big\}
 \]
and
 \[
  E_m (x) := \big\{ k \in \Z^d : |\det(A)|^{m-1} < 1+\rho_A (A^{-j} x - k) \leq |\det(A)|^{m} \big\}.
 \]
Then
\begin{align*}
 (c_j^A(x) )^r &\lesssim \sum_{m \in \mathbb{N}_0} |\det(A)|^{-m \lambda} \sum_{k \in E_m (x) } |c_{j,k} |^r \\
 &\leq \sum_{m \in \mathbb{N}_0} |\det(A)|^{-m \lambda} \bigg( \sum_{k \in E_m (x) } |c_{j,k} |^a \bigg)^{r/a} \\
 &\leq \eta^{-r/a} \sum_{m \in \mathbb{N}_0} |\det(A)|^{-m \lambda}  |\det(A)|^{-jr/a} \bigg(  \sum_{k \in E_m (x) } \int_{E_{j,k}} \sum_{k \in \Z^d} |c_{j,k} |^a \1_{E_{j,k}} (y) \; dy \bigg)^{r/a} \\
 &\leq \eta^{-r/a} \sum_{m \in \mathbb{N}_0} |\det(A)|^{-m \lambda} |\det(A)|^{-jr/a}  \bigg(  \int_{B_m} \sum_{k \in \Z^d} |c_{j,k} |^a \1_{E_{j,k}} (y) \; dy \bigg)^{r/a},
\end{align*}
where the last inequality used that $$\bigcup_{k \in E_m (x)} E_{j,k} \subseteq \bigcup_{k \in E_m (x)} Q_{j,k} \subseteq B_m := B_{\rho_A} (x, C |\det(A)|^{j+m})$$ for some constant $C>0$. Hence, using the definition of the Hardy-Littlewood maximal function, we obtain
\begin{align*}
 (c_j^A(x) )^r &\lesssim \eta^{-r/a} \sum_{m \in \mathbb{N}_0} |\det(A)|^{-m \lambda} |\det(A)|^{-jr/a} |B_m|^{r/a} \bigg( \frac{1}{|B_m|} \int_{B_m} \sum_{k \in \Z^d} |c_{j,k} |^a \1_{E_{j,k}} (y) \; dy \bigg)^{r/a} \\
 &\lesssim \eta^{-r/a} \sum_{m \in \mathbb{N}_0} |\det(A)|^{-m (\lambda - r/a)}   \bigg( M_{\rho_A} \bigg( \sum_{k \in \Z^d} |c_{j,k} |^a \1_{E_{j,k}} \bigg)(x) \bigg)^{r/a} \\
 &\lesssim\eta^{-r/a} \bigg( M_{\rho_A} \bigg( \sum_{k \in \Z^d} |c_{j,k} |^a \1_{E_{j,k}} \bigg)(x) \bigg)^{r/a},
\end{align*}
which settles the claim.
\end{proof}

Lastly, we will use the Fefferman-Stein vector-valued inequalities. The following result is a special case of, e.g., \cite[Theorem 1.2]{grafakos2009vector}.

\begin{lem}\label{sfvvi}
Let $A \in \mathrm{GL}(d, \mathbb{R})$ be expansive. Suppose that $1<p<\infty$, $1<q \le \infty$.  Then
there exists a constant $C>0$ such that
\begin{displaymath}
\bigg\|\bigg(\sum_{i}| M_{\qn}f_{i}|^{q}\bigg)^{1/q}\bigg\|_
{L^{p}(\R^d)}\quad\le C\bigg\|\bigg(\sum_{i} |f_{i}|^{q}\bigg)^
{1/q}\bigg\|_{L^{p}(\R^d)}
\end{displaymath}
holds for any $(f_{i})_i \subset L^{p}(\R^d)$.
\end{lem}

\subsection{Case of equal determinants}
In this subsection, we treat the case of matrices with equal determinant.

\begin{lem} \label{sc}
Suppose that  $A, B \in \mathrm{GL}(d, \R)$ are expansive dilations such that $A$ and $B$ are equivalent and $|\det (A)| = |\det (B)|$. Then, for each $j\in \Z$, there exists a bijection $\pi_j: \Z^d \to \Z^d$ such that for any $(c_{j,k})\in \mathbb C^{\Z \times \Z^d}$, letting $s=(s_{j,k})$ denote the sequence satisfying $s_{j,\pi_j(k)} = c_{j,k}$ for $j\in \Z$, $k\in\Z^d$, we have
\[
c^{A}_j(x) \asymp s^{B}_j(x) \qquad \text{for all } j\in \Z, \; x\in \R^d,
\]
with equivalence constants depending only on $A$ and $B$.
\end{lem}

\begin{proof}
Throughout the proof, we let $\rho_A, \rho_B$ be symmetric homogeneous quasi-norms.

Let $j\in \Z$ be arbitrary. By Lemma \ref{bd5}, there exists a bijection $\pi_j: \Z^d \to \Z^d$ such that
\[
\sup_{k \in \mathbb{Z}^d} \| k - A^{-j} B^j\pi_j (k) \| \le \sqrt{d} (1+ ||A^{-j} B^j||).
\]
Since the matrices $A$ and $B$ are equivalent, we have that $\sup_{j\in \Z} ||A^{-j} B^j|| <\infty$. Hence, there exists a constant $C>0$ such that
\[
\sup_{k \in \mathbb{Z}^d} \rho_A( k - A^{-j} B^j\pi_j (k) ) \le C \qquad\text{for all }j\in \Z.
\]
Using the homogeneity of $\rho_A$, we get
\[
|\det(A)|^{-j} \rho_A(A^j k - B^j \pi_j(k)) \le C.
\]
Given $x \in \R^d$, an application of  the triangle inequality gives
\[
\rho_A(x-A^jk)  \le c\big( \rho_A(x-B^j\pi_j(k))+\rho_A(A^j k - B^j \pi_j(k)) \big).
\]
Hence,
\[
|\det(A)|^{-j} \rho_A(x-A^jk) \le c (|\det(A)|^{-j} \rho_A(x-B^j\pi_j(k)) +C).
\]
A similar argument shows
\[
 |\det(A)|^{-j} \rho_A (x-B^j \pi_j(k)) \leq c (|\det(A)|^{-j} \rho_A (x - A^j k) + C).
\]
Consequently,
\begin{align*}
1+ |\det(A)|^{-j} \rho_A(x-A^jk) &\asymp 1+ |\det(A)|^{-j} \rho_A(x-B^j\pi_j(k)) \\
&\asymp 1+ |\det(B)|^{-j} \rho_B(x-B^j\pi_j(k)),
\end{align*}
and thus
\[
\begin{aligned}
(c^{A}_j(x) )^r &=  \sum_{k\in \Z^d} \frac{|c_{j,k}|^r}{(1+\rho_A(A^{-j}x - k))^\lambda} \\
& \asymp \sum_{k\in \Z^d} \frac{|c_{j,k}|^r}{(1+\rho_B(B^{-j}x - \pi_j(k)))^\lambda} \\
& = \sum_{k\in \Z^d} \frac{|s_{j,\pi_j(k)}|^r}{(1+\rho_B(B^{-j}x - \pi_j(k)))^\lambda} \\
&= (s^B_j(x))^r,
\end{aligned}
\]
as desired.
\end{proof}

Using the previous lemma, we now prove the following theorem.

\begin{thm} \label{main}
Let $\alpha \in \mathbb{R}$ and $p,q \in (0, \infty]$.
Suppose that $A, B \in \mathrm{GL}(d, \R)$ are expansive matrices that are equivalent and satisfy $|\det (A)|=|\det (B)|$.
Then, for each $j\in \Z$, there exists a bijection $\pi_j: \Z^d \to \Z^d$ such that 
an induced permutation operator $P: \mathbb C^{\Z \times \Z^d} \to \mathbb C^{\Z \times \Z^d}$ respecting scales, which is defined by
\[
P( (c_{j,k})_{j\in \Z, k\in \Z^d}) = (c_{j,\pi_j^{-1}(k)})_{j\in \Z, k\in \Z^d},
\]
is an isomorphism between between the spaces $\dot {\mathbf f}^\alpha_{p,q}(A)$ and $\dot {\mathbf f}^\alpha_{p,q}(B)$. That is, $P$ is onto and
\begin{equation}\label{main2}
||P(c)||_{\dot {\mathbf f}^\alpha_{p,q}(B)} \asymp ||c||_{\dot {\mathbf f}^\alpha_{p,q}(A)}
\end{equation}
for all $c\in \mathbb C^{\Z \times \Z^d}$.
\end{thm}

\begin{proof}
Assume first that $p<\infty$.
Take $a>0$ such that $p/a>1$ and $q/a>1$. Let $r=a$ and $\lambda>r/a=1$.
 Let $(c_{j,k})\in \mathbb C^{\Z \times \Z^d}$ be arbitrary.
Under the assumptions, an application of Lemma \ref{sc} yields a bijection $\pi_j : \mathbb{Z}^d \to \mathbb{Z}^d$ such that the sequence $s_{j,k} = c_{j,\pi_j^{-1}(k)}$ for $j\in \Z$, $k\in\Z^d$ satisfies $c_j^A (x) \asymp s_j^B (x)$ for all $j \in \mathbb{Z}$ and $x \in \mathbb{R}^d$. A combination of the above gives
\[
\begin{aligned}
||c||_{\TLsA} & = \bigg\| \bigg( \sum_{j\in \Z} \sum_{k\in \Z^d} ( |\det(A)|^{-j(\alpha+1/2)}|c_{j,k}| \1_{A^j ([0, 1]^d + k)})^q \bigg)^{1/q} \bigg\|_{L^p} \\
& \lesssim \bigg( \int_{\R^d} \bigg( \sum_{j\in \Z} |\det(A)|^{-jq(\alpha+1/2)} |c^A_j(x)|^q \bigg)^{p/q} dx \bigg)^{1/p}  \\
& \asymp \bigg( \int_{\R^d} \bigg( \sum_{j\in \Z}  |\det(B)|^{-jq(\alpha+1/2)}|s^B_j(x)|^q \bigg)^{p/q} dx \bigg)^{1/p}
\\
&\lesssim
\bigg\| \bigg( \sum_{j\in \Z}|\det(B)|^{-jq(\alpha+1/2)}  \bigg( M_{\rho_B} \bigg( \sum_{k\in \Z^d} |s_{j,k}| \1_{B^j([0,1]^d+k)}\bigg)^a \bigg)^{q/a} \bigg)^{a/q} \bigg\|_{L^{p/a}}^{1/a}.
\end{aligned}
\]
Since $p/a>1$ and $q/a>1$, an application of Lemma \ref{sfvvi} allows us  to obtain the bound
$||c||_{\TLsA} \lesssim ||s||_{\TLsB}$. By a symmetric argument, we obtain $||s||_{\TLsB} \lesssim ||c||_{\TLsA}$. This completes the proof in the case $p<\infty$. The proof in the case $p=\infty$ is a simple modification of this argument using Theorem \ref{pinf}.
\end{proof}

\subsection{General case}

Using a similar approach as in the proof of Theorem \ref{main}, we can show its variant when $|\det(A) | \ne |\det(B)|$. 
For this, we need a variant of Lemma \ref{sc} dealing with the case $|\det(A)| > |\det(B)|$. The opposite case $|\det(A)| < |\det(B)|$ follows automatically by symmetry.

\begin{lem} \label{scd}
Suppose that  $A, B \in \mathrm{GL}(d, \R)$ are expansive matrices that are equivalent and satisfy $|\det (A)| > |\det (B)|$. Let $\varepsilon = \ln |\det(A)| /\ln |\det(B)|$. 
Then, for each $j\in \Z$, there exists an injection $\pi_{j}: \Z^d \to \Z^d$ with the following property. For any $c=(c_{j,k}) \in \mathbb C^{\Z \times \Z^d}$, we let $(s_{i,k})_{i\in \Z, k\in \Z^d}$ denote the sequence 
\begin{equation}\label{scd1}
s_{i,k} =
\begin{cases}  c_{j ,\pi_j ^{-1}(k)} & \text{if } i=\lfloor j \varepsilon \rfloor \text{ and } k\in \pi_j(\Z^d), \\
0 & \text{otherwise}.
\end{cases}
\end{equation}
Then, for any $j\in \Z$ and $i=\lfloor j \varepsilon \rfloor$, we have
\[
c^{A}_j(x) \asymp s^{B}_i(x) \qquad \text{for all } \quad x\in \R^d,
\]
with equivalence constants depending only on $A$ and $B$.
\end{lem}

\begin{proof}
Since $A$ and $B$ are equivalent, it follows from Lemma \ref{lem:equivalent} that
\begin{equation}\label{inj5}
\sup_{j \in \mathbb{Z}} \| A^{-j} B^{\lfloor \varepsilon j \rfloor} \| < \infty.
\end{equation}
Let $j\in \Z$ be arbitrary and let $i=\lfloor j \varepsilon \rfloor$. Note that $|\det(A^j)| \geq |\det(B^i)|$ as $|\det(B)| > 1$.
By Lemma \ref{bd5}, there exists therefore an injection $\pi_j: \Z^d \to \Z^d$ such that
\[
\sup_{k \in \mathbb{Z}^d} \| k - A^{-j} B^{i} \pi_j (k) \| \le \sqrt{d} (1+ ||A^{-j} B^{i}||).
\]
By \eqref{inj5} there exists a constant $C>0$ such that
\[
\sup_{k \in \mathbb{Z}^d} \rho_A( k - A^{-j} B^{i} \pi_j (k) ) \le C \qquad\text{for all }j\in \Z.
\]
For convenience, take a symmetric quasi-norm satisfying $\rho_A(-x) = \rho_A (x)$ for all $x\in \R^d$.
Using the homogeneity of $\rho_A$, we get
\[
|\det(A)|^{-j} \rho_A(A^j k - B^i \pi_j(k)) \le C.
\]
Given $x \in \R^d$, an application of  the triangle inequality gives
\[
\rho_A(x-A^jk)  \le c\big( \rho_A(x-B^i\pi_j(k))+\rho_A(A^j k - B^i \pi_j(k)) \big)
\]
and thus
\[
|\det(A)|^{-j} \rho_A(x-A^jk) \le c (|\det(A)|^{-j} \rho_A(x-B^i \pi_j(k)) +C).
\]
Similarly, it is shown that
\[
 |\det(A)|^{-j} \rho_A (x-B^i \pi_j(k)) \leq c (|\det(A)|^{-j} \rho_A (x - A^j k) + C).
\]
Therefore,
\begin{align*}
1+ |\det(A)|^{-j} \rho_A(x-A^jk) &\asymp 1+ |\det(A)|^{-j} \rho_A(x-B^i\pi_j(k)) \\
&\asymp 1+ |\det(B)|^{-i} \rho_B(x-B^i \pi_j(k)),
\end{align*}
which yields
\[
\begin{aligned}
(c^{A}_j(x) )^r &=  \sum_{k\in \Z^d} \frac{|c_{j,k}|^r}{(1+\rho_A(A^{-j}x - k))^\lambda} \\
& \asymp \sum_{k\in \Z^d} \frac{|c_{j,k}|^r}{(1+\rho_B(B^{-i}x - \pi_j(k)))^\lambda} \\
& = \sum_{k\in \Z^d} \frac{|s_{i,\pi_j(k)}|^r}{(1+\rho_B(B^{-i}x - \pi_j(k)))^\lambda} \\
&= (s^B_i (x))^r.
\end{aligned}
\]
This completes the proof.
\end{proof}

\begin{thm} \label{Main}
Let $\alpha \in \mathbb{R}$ and $p,q \in (0, \infty]$.
Suppose that $A, B \in \mathrm{GL}(d, \R)$ are expansive matrices that are equivalent and satisfy $|\det(A)|>|\det(B)|$.
Let $\varepsilon = \ln |\det(A)| /\ln |\det(B)|$. For each $j\in \Z$, there exists an injection $\pi_j: \Z^d \to \Z^d$, such that induced operators $S$ and $T$ on  $\mathbb C^{\Z \times \Z^d}$, which are defined by
\[
\begin{aligned}
S( (c_{j,k})_{j\in \Z, k\in \Z^d}) & = (s_{i,k})_{i\in \Z, k\in \Z^d}, \qquad \text{where }s_{i,k} =
\begin{cases}  c_{j ,\pi_j ^{-1}(k)} & \text{if }i=\lfloor j \varepsilon \rfloor \text{ and } k\in \pi_j(\Z^d),\\
0 & \text{otherwise},
\end{cases}
\\
T( (s_{i,k})_{i\in \Z, k\in \Z^d}) & = (c_{j,k})_{j\in \Z, k\in \Z^d},
\qquad \text{where } c_{j,k} =  s_{i,\pi_j(k)}, \ i=\lfloor j \varepsilon \rfloor, k\in \Z^d,
 \end{aligned}
\]
satisfy:
\begin{enumerate}
\item $S: \TLsA \hookrightarrow \TLsB$ is a bounded  embedding, 
\item $T: \TLsB \to \TLsA$ is bounded and surjective, and 
\item $T \circ S$ is the identity. 
\end{enumerate}
\end{thm}

That is, the following diagram commutes:
\begin{center}
\begin{tikzcd}[column sep=small]
& \TLsB \arrow[dr, "T"]& \\
\TLsA \arrow[ur, hook, "S"] \arrow[rr, "id"] &  & \TLsA\\
\end{tikzcd}
\end{center}

\begin{proof}
Assume first that $p<\infty$. Since $A$ and $B$ are equivalent, by Lemma \ref{lem:equivalent}, we have
\[
\sup_{j \in \mathbb{Z}} \| A^{-j} B^{\lfloor \varepsilon j \rfloor} \| < \infty.
\]
Take $a>0$ such that $p/a>1$ and $q/a>1$. Let $r=a$ and $\lambda>r/a=1$. 
For any $j\in \Z$, let $\pi_j: \Z^d \to \Z^d$ be an injection as in Lemma \ref{scd}.

We first prove that $T: \TLsB \to \TLsA$ is bounded. For this, let $s=(s_{i,k}) \in \mathbb C^{\Z \times \Z^d}$ be arbitrary and
 define the sequence $c=(c_{j,k})\in \mathbb C^{\Z \times \Z^d}$  by $c=T(s)$.
 Using the estimate $|\det(A)|^{-j} \leq |\det(B)|^{-\lfloor \varepsilon j \rfloor }$ for $j \in \mathbb{Z}$, together with \eqref{main5}, Lemma \ref{max}, and Lemma \ref{scd}, gives
\[
\begin{aligned}
||c||_{\TLsA} & = \bigg\| \bigg( \sum_{j\in \Z} \sum_{k\in \Z^d} ( |\det(A)|^{-j(\alpha+1/2)}|c_{j,k}| \1_{A^j ([0, 1]^d + k)})^q \bigg)^{1/q} \bigg\|_{L^p} \\
& \lesssim \bigg( \int_{\R^d} \bigg( \sum_{j\in \Z} |\det(A)|^{-jq(\alpha+1/2)} |c^A_j(x)|^q \bigg)^{p/q} dx \bigg)^{1/p}  \\
& \lesssim \bigg( \int_{\R^d} \bigg( \sum_{j\in \Z}  |\det(B)|^{-{\lfloor \varepsilon j \rfloor} q(\alpha+1/2)}|s^B_{\lfloor \varepsilon j \rfloor}(x)|^q \bigg)^{p/q} dx \bigg)^{1/p}
\\
& \le \bigg( \int_{\R^d} \bigg( \sum_{i\in \Z}  |\det(B)|^{-i q(\alpha+1/2)}|s^B_i (x)|^q \bigg)^{p/q} dx \bigg)^{1/p}
\\
&\lesssim
\bigg\| \bigg( \sum_{i\in \Z}|\det(B)|^{-i q(\alpha+1/2)}  \bigg( M_{\rho_B} \bigg( \sum_{k\in \Z^d} |s_{i,k}| \1_{B^j([0,1]^d+k)}\bigg)^a \bigg)^{q/a} \bigg)^{a/q} \bigg\|_{L^{p/a}}^{1/a}.
\end{aligned}
\]
The penultimate step used that $\varepsilon >1$, which implies that the sum over $j\in \Z$ appears as a part of the sum over all $i\in \Z$.
Since $p/a>1$ and $q/a>1$, an application of Lemma \ref{sfvvi} allows us  to obtain the bound
$||c||_{\TLsA} \lesssim ||s||_{\TLsB}$. This proves that $T: \TLsB \to \TLsA$ is bounded. 

We next show that $S: \TLsA \to \TLsB$ is bounded. Let $c=(c_{j,k})\in \mathbb C^{\Z \times \Z^d}$ and
 define the sequence $s=(s_{i,k}) \in \mathbb C^{\Z \times \Z^d}$  by $s=S(c)$.  A combination of \eqref{main5}, the definition of the operator $S$, Lemma \ref{max}, and Lemma \ref{scd} yields
 \[
\begin{aligned}
||s||_{\TLsB} & = \bigg\| \bigg( \sum_{i\in \Z} \sum_{k\in \Z^d} ( |\det(B)|^{-i(\alpha+1/2)}|s_{i,k}| \1_{B^i ([0, 1]^d + k)})^q \bigg)^{1/q} \bigg\|_{L^p} \\
& \lesssim \bigg( \int_{\R^d} \bigg( \sum_{i\in \Z} |\det(B)|^{-iq(\alpha+1/2)} |s^B_i(x)|^q \bigg)^{p/q} dx \bigg)^{1/p}  \\
& = \bigg( \int_{\R^d} \bigg( \sum_{j\in \Z}  |\det(B)|^{-{\lfloor \varepsilon j \rfloor} q(\alpha+1/2)}|s^B_{\lfloor \varepsilon j \rfloor}(x)|^q \bigg)^{p/q} dx \bigg)^{1/p}
\\
& \asymp \bigg( \int_{\R^d} \bigg( \sum_{j\in \Z}  |\det(A)|^{-j q(\alpha+1/2)}|c^A_j (x)|^q \bigg)^{p/q} dx \bigg)^{1/p}
\\
&\lesssim
\bigg\| \bigg( \sum_{j\in \Z}|\det(A)|^{-j q(\alpha+1/2)}  \bigg( M_{\rho_A} \bigg( \sum_{k\in \Z^d} |c_{j,k}| \1_{A^j([0,1]^d+k)}\bigg)^a \bigg)^{q/a} \bigg)^{a/q} \bigg\|_{L^{p/a}}^{1/a},
\end{aligned}
\]
where the penultimate step also used that $|\det(A)|^{-j} \asymp |\det(B)|^{-\lfloor \varepsilon j \rfloor}$ for all $j \in \mathbb{Z}$.
Hence, by an application of Lemma \ref{sfvvi}, we have that $||s||_{\TLsB} \lesssim ||c||_{\TLsA}$, which proves that $S: \TLsA \to \TLsB$ is bounded.

The fact that $T \circ S$ is the identity follows from the definitions of $T$ and $S$. Consequently, $T$ is surjective. Finally, the operator $S$ is an embedding, since for any $c\in \TLsA$ we have 
\[
||c||_{\TLsA} = ||T(Sc)||_{\TLsA} \lesssim ||Sc||_{\TLsB} \lesssim ||c||_{\TLsA}.
\]
This completes the proof in the case $p<\infty$. The proof in the case $p=\infty$ is a simple modification of this argument using Theorem \ref{pinf}.
\end{proof}

\section*{Acknowledgements}
The authors thank Felix Voigtlaender for helpful discussions on the results in Section 5.
The first author was partially supported by the NSF grant DMS-2349756.
For the second author, this research was funded in whole or in part by the Austrian
Science Fund (FWF):  10.55776/PAT2545623.  For open access purposes,
the authors have applied a CC BY public copyright license to any author-accepted manuscript
version arising from this submission.

\end{document}